\documentclass{amsart}
\usepackage{graphicx}
\usepackage{amsmath}
\usepackage{amsfonts}
\usepackage{amssymb}
\usepackage{ifpdf}
\newtheorem{theorem}{Theorem}

\newtheorem{definition}[theorem]{Definition}
\newtheorem{example}[theorem]{Example}

\newtheorem{remark}[theorem]{Remark}

\makeatletter \@namedef{subjclassname@2020}{\textup{2020}
Mathematics Subject Classification} \makeatother

\begin{document}

\title[]{A note on rearrangement Poincar\'e inequalities and the doubling condition}
\author{Joaquim Mart\'{i}n$^{\ast}$}
\address{Department of Mathematics\\
Universitat Aut\`{o}noma de Barcelona}
\email{Joaquin.Martin@uab.cat\\\emph{ORCID: 0000-0002-7467-787X}}
\author{Walter A.  Ortiz**}
\address{Department of Mathematics\\
Universitat Aut\`{o}noma de Barcelona}
\email{WalterAndres.Ortiz@uab.cat\\\emph{ORCID:
0000-0002-8617-3919}}
\thanks{$^{\ast}$Partially supported by Grants PID2020-113048GB-I00 and PID2020-114167GB-I00 funded  both by MCIN/AEI/10.13039/501100011033 and Grant 2021-SGR-00071 (AGAUR, Generalitat de Catalunya)}
\thanks{**Partially supported by Grant 2021-SGR-00071 (AGAUR, Generalitat de Catalunya)}
\thanks{This paper is in final form and no version of it will be submitted for
publication elsewhere.} 
\thanks{Conflict of Interest: The authors declare that they have no conflict
of interest.}

\begin{abstract}
We introduce Poincar\'{e} type inequalities based on rearrangement
invariant spaces in the setting of metric measure spaces and analyze
when they imply the doubling condition on the underline measure.
\end{abstract}
\maketitle

\section{Introduction}
An important topic of study in the rich theory of Sobolev spaces
defined on metric measure spaces, are the so-called spaces that
support Poincar\'{e} inequalities introduced in \cite{HK} (for more
information, see for example \cite{FLW}, \cite{FPW}, \cite{Ha},
\cite {HaKo}, \cite{Ko}, \cite{KMR}, \cite{Sh}, \cite{DJS},
\cite{AT}, \cite{MM1}, \cite{MM2}, \cite{MO1} and \cite{MO2} the
references quoted therein).

In order to define this notion, let us recall that a metric-measure
space $(\Omega ,d,\mu )$ is a metric
space $(\Omega ,d)$ with a Borel measure $\mu $ such that $0<\mu (B)<\infty $%
, for every ball $B$ in $\Omega $, we denote by $\sigma B$ the
dilation of a ball $B$ by the factor $\sigma ,$ i.e. if
$B=B(x,r)=\left\{ y\in \Omega :d(x,y)<r\right\} $, then $\sigma
B:=B(x,\sigma r).$ A Borel $\mu $-measurable function $g\geq 0$ is
an upper gradient of $f,$ if for all rectifiable curves $\gamma $
joining points $x,$ $y$ in $\Omega $, we have
\begin{equation*}
|f(x)-f(y)|\leq \int_{\gamma }gds.
\end{equation*}
This definition of upper gradient is due to Heinonen and Koskela in
\cite{HK} (see also \cite{HKST} for a more detailed exposition).

A metric measure space $(\Omega ,d,\mu )$ is said to support a
$\left( q,p\right) -$Poincar\'{e} inequality, $q,p\in \lbrack
1,\infty ),$ if there is a constant $c>0$ and $\sigma \geq 1$ such
that
\begin{equation}
\left( \frac{1}{\mu (B)}\int_{B}\left| f-f_{B}\right| ^{q}d\mu
\right) ^{1/q}\leq cr\left( \frac{1}{\mu (\sigma B)}\int_{\sigma
B}\left| g\right| ^{p}d\mu \right) ^{1/p}  \label{qpPoin}
\end{equation}
whenever $B$ is a ball of radius $r>0$, $f\in L_{loc}^{1}(\Omega )$ and $g\ $%
is an upper gradient of $f$ (here $f_{B}$ denotes the integral average: $%
f_{B}=\frac{1}{\mu (B)}\int_{B}fd\mu ).$

Usually, at the root of analysis in this field, doubling measures
play a central role since they provide a homogeneous space
structure, which makes it possible to adapt many classical tools
available in the Euclidean space. Recall that $\mu$ is said to be
doubling provided there exists a constant $C=C_{\mu}>0$ such that
\begin{equation*}
\mu (2B)\leq C\mu (B)\text{ for all balls }B\subset \Omega .
\end{equation*}

For example it was proved in \cite[Theorem 5.1]{HaKo} that a $\left(
1,p\right) -$Poincar\'{e} inequality self-improves in the sense
that it implies a $\left( q,p\right) -$Poincar\'{e} inequality for some $%
q\in (p,\infty ).$

Recently, the study of when the family of inequalities
(\ref{qpPoin}) implies that the underlying measure is doubling has
begun to be widely considered (see for example \cite{Ha}, \cite{Sh},
\cite{FPW}, \cite{HKST}, \cite{HaKo} and the references quoted
therein, for more applications).

In \cite{KMR},
it has been shown that, if a Borel measure $\mu $ in the Euclidean space $(%
\mathbb{R}^{n},d),$ satisfies the following $\left( q,p\right) -$
Poincar\'{e} type inequality ($q>p\geq 1)$: There exists a constant
$C>0$ such that for every Euclidean ball $B$ and every Lipschitz
function $\varphi $ compactly supported on $B$ it holds true that
\begin{equation*}
\left( \frac{1}{\mu (B)}\int_{B}\left| \varphi \right| ^{q}d\mu
\right) ^{1/q}\leq Cr\left( \frac{1}{\mu (B)}\int_{B}\left| \nabla
\varphi \right| ^{p}d\mu \right) ^{1/p}+C\left( \frac{1}{\mu
(B)}\int_{B}\left| \varphi \right| ^{p}d\mu \right) ^{1/q},
\end{equation*}
then $\mu $ is doubling. In \cite{AH}, a perusal at the proof of the
previous result, allowed the authors to translate this beautiful
result to the metric setting without assuming the balls had the same
radius on both sides.

It is know (see for example \cite[Example 2.2]{Ko} and \cite[Example
4]{AH}) that there exits metric-measure spaces endowed with a
non-doubling measure
that supports a $(p,p)-$Poincar\'{e} inequality (\ref{qpPoin}) for all $%
1\leq p<\infty ,$ therefore the result obtained in \cite{AH} is
sharp when considering Poincar\'{e} inequalities in which the
functions norms involved are given by $L^{p}-$spaces.

The natural question, in view of the above, is whether we can obtain
weaker versions of Poincar\'{e}'s inequalities that still would
imply the doubling property. In other words, we wonder if it is
possible to replace the norm on
the left-hand side by one smaller than any power bump and still obtain that $%
\mu $ is doubling. A first attempt in this direction was done by L.
Korobenko in \cite[Theorem 2.4]{Ko}, where she considered in the
left hand
side an Orlicz norm\footnote{%
See Section \ref{secpreli} below.} and a $L^{1}$ norm in the right
hand side and was able to prove the following result: Let $(\Omega
,d,\mu )$ be a metric measure space\ and let $\Phi $ be a Young
function (see Section \ref {secpreli} below) such that there exists
an $\alpha >1$ satisfying
\begin{equation}
\Phi (t)\geq t(1+\ln t)^{\alpha },\text{ \ }t>1,  \label{condi}
\end{equation}
for some $\alpha >1.$ Assume that there is a constant $C>0,$ such
that
\begin{eqnarray}
\left\| w\right\| _{L^{\Phi }\left( B;\frac{d\mu }{\mu (B)}\right) }
&:&=\inf \left\{ \lambda >0;\int_{B}A\left( \frac{|w(x)|}{\lambda
}\right)
\frac{d\mu }{\mu (B)}\leq 1\right\}  \label{poinOrli} \\
&\leq &Cr\int_{B}\left| w\right| ^{p}\frac{d\mu }{\mu (B)} .  \notag
\end{eqnarray}
whenever $B$ is a ball, $w$ is a $\mu $-measurable function
supported on $B$ with zero boundary values and $g$ is an
upper-gradient of $w$, then $\mu $ is doubling.

The above two results raise the question of what kind of function
spaces are suitable in Poincar\'{e}'s inequalities definition in
order obtain the doubling condition. The aim of this paper will try
to answer this question, to this end, we will enlarge the class of
function spaces, and corresponding norms, which define
Poincar\'{e}'s inequalities and will analyze when still imply the
doubling condition. Since the norm of functions in Lebesgue or
Orlicz spaces just only depends on its integrability properties, the
natural
class seems to be the class of rearrangement invariant function spaces (%
\emph{r.i. spaces} for short). Roughly speaking, a r.i. space is a
Banach function space where the norm of a function depends only on
the $\mu $-measure of its level sets. Lebesgue and Orlicz spaces are
examples of r.i. spaces (see Section \ref{secpreli} below).

The paper is organized as follows. In Section 2 we introduce the
notation, the standard assumptions, give a brief overview in the
theory of r.i. spaces used in the paper and state our main result.
Section 3 is dedicated to the proof of this result and we make
further comments and remarks. Finally, the last section contains
several examples.

\section{Preliminaries\label{secpreli}}

In this section we establish some further notation and background
information, we provide more details about metrics spaces and r.i.
spaces in will be working with and state our main result.

For measurable functions $f:\Omega \rightarrow \mathbb{R},$ the
distribution function of $f$ is given by
\begin{equation*}
\mu _{f}(t)=\mu \{x\in {\Omega }:\left| f(x)\right| >t\}\ \ \ \
(t>0).
\end{equation*}

The \textbf{decreasing rearrangement} $f_{\mu }^{\ast }$ of $f$ is
the right-continuous non-increasing function from $[0,\infty )$ into
$[0,\infty ) $ which is equimeasurable with $f$. Namely,
\begin{equation*}
f_{\mu }^{\ast }(s)=\inf \{t\geq 0:\mu _{f}(t)\leq s\}.
\end{equation*}
We say that a Banach function space with the Fatou property
$X=X({\Omega })$ on $({\Omega },d,\mu )$ is a \textbf{r.i. space},
if $g\in X$ implies that all equimeasurable function $f$, i.e.
$f_{\mu }^{\ast }=g_{\mu }^{\ast },$ also belong to $X,$ and $\Vert
f\Vert _{X}=\Vert g\Vert _{X}$.

A basic example of r.i. spaces are the standard Lebesgue spaces $%
L^{p}(\Omega )$, for $p\geq 1.$ A generalization of the Lebesgue
spaces is provided by the Orlicz spaces.

Let $A\colon \lbrack 0,\infty )\rightarrow \lbrack 0,\infty )$ a
Young function, namely a convex (non trivial), left-continuous
function vanishing at $0,$ the\textbf{\ Orlicz space} $L^{A}(\Omega
,\mu )$ is the collection of all $\mu -$measurable functions $f$ for
which there exists a $\lambda $ such that
\begin{equation*}
\int_{\Omega }A\left( \frac{|f(x)|}{\lambda }\right) d\mu <\infty .
\end{equation*}
The Orlicz space $L^{A}(\mu )$ is endowed with the Luxemburg norm

\begin{equation*}
\left\| f\right\| _{L^{A}(\Omega ,\mu )}=\inf \left\{ \lambda
>0;\int_{\Omega }A\left( \frac{|f(x)|}{\lambda }\right) d\mu \leq 1\right\} .
\end{equation*}

A r.i. space $X=X({\Omega })$ on $({\Omega },d,\mu )$ can be
represented by
a r.i. space on the interval $(0,\mu (\Omega )),$ with Lebesgue measure, $%
\bar{X}=\bar{X}(0,\mu (\Omega ))$ such that $\Vert f\Vert _{X}=\Vert
f_{\mu }^{\ast }\Vert _{\bar{X}},$ for every $f\in X.$

The space $\bar{X}$ is called the representation spaces of $X,$ a
characterization of the norm $\Vert \cdot \Vert _{\bar{X}}$ is given
in \cite[Theorem 4.10 and subsequent remarks]{BS}.

\begin{remark}
\label{carac}Since (see \cite{BS}) for any nonnegative, increasing
left continuous function $\psi :[0,\mu (\Omega ))$ such that $\psi
(0^{+})=0$ we have that
\begin{equation*}
\int_{\Omega }\psi (\left| f\right| )d\mu =\int_{0}^{\mu (\Omega
)}\psi (f_{\mu }^{\ast }(s))ds,
\end{equation*}
it can be easily seen that the representation space of an Orlicz space $%
L^{\Phi }(\Omega ,\mu )$ is $L^{\Phi }(\left( 0,\mu (\Omega )\right)
),$ the Orlicz space defined on $\left( 0,\mu (\Omega )\right) $
with respect
then Lebesgue measure, in particular, the representation space of $%
L^{p}(\Omega )$ is $L^{p}(\left( 0,\mu (\Omega )\right) ).$
\end{remark}

We now look for a form of Poincar\'{e}'s inequalities defined by
r.i. spaces.

\begin{definition}
Let $({\Omega },d,\mu )$ be a metric measure space, let $B\subset
\Omega $
be a ball , denote $%
{\mu}%
_{B}=%
{\mu}%
(B)^{-1}%
{\mu}%
.$ Given a r.i. space $X$ on $\Omega $ we define
\begin{equation*}
X(B,%
{\mu}%
_{B})=\left\{ f\in L_{loc}^{1}(\Omega ):\left\| f\right\| _{X(B,%
{\mu}%
_{B})}:=\left\| \left( f\chi _{B}\right) _{\mu }^{\ast }(s\mu (B))\right\| _{%
\bar{X}(0,1)}<\infty \right\} .
\end{equation*}
($\bar{X}(0,1)$ denotes the representation spaces of $X$ where its
function
norm is restricted to $(0,1)$)\footnote{%
The norm $\left\| \cdot \right\| _{X(B,%
{\mu}%
_{B})}$ is closely related with the average norm introduced in
\cite{CaCi}.}.
\end{definition}

\begin{definition}
\label{XYPoin}Let $\left( \Omega ,d,\mu \right) $ be as above. Given two $%
X,Y $ two r.i. spaces on $\Omega ,$ we say that the triple $\left(
\Omega ,d,\mu \right) $ admits a $\left( X,Y\right) -$Poincar\'{e}
inequality if there exist constants $C_{S},$ $\sigma \geq 1$ such
that
\begin{equation}
\left\| f-f_{B}\right\| _{X(B,%
{\mu}%
_{B})}\leq C_{S}r\left\| g\right\| _{Y\left( \sigma B,\mu _{\sigma
B}\right) },  \label{poin}
\end{equation}
whenever $B$ is a ball of radius $r\in (0,\infty ),$ $f\in
L_{loc}^{1}(\Omega )$ and $g:\Omega \rightarrow \lbrack 0,\infty ]$
is an upper gradient of $f.$
\end{definition}

\begin{remark}
Let $\Phi $ be a Young function an $h\in L_{loc}^{1}(\Omega ).$ By
remark \ref{carac} we get
\begin{eqnarray*}
\int_{B}\Phi \left( \left| h\right| \right) \frac{d\mu }{\mu (B)}
&=&\int_{\Omega }\Phi \left( h\chi _{B}\right) \frac{d\mu }{\mu (B)} \\
&=&\int_{0}^{\mu (B)}\Phi \left( \left( h\chi _{B}\right) _{\mu
}^{\ast
}(s)\right) \frac{ds}{\mu (B)} \\
&=&\int_{0}^{1}\Phi \left( \left( h\chi _{B}\right) _{\mu }^{\ast
}(s\mu (B))\right) ds.
\end{eqnarray*}
Thus
\begin{equation*}
\left\| h\right\| _{L^{\Phi }(B,%
{\mu}%
_{B})}=\left\| \left( h\chi _{B}\right) _{\mu }^{\ast }(s\mu
(B))\right\| _{L^{\Phi }(0,1)}.
\end{equation*}
In Particular, our $\left( L^{q},L^{p}\right) -$Poincar\'e
inequality coincides with the classical $(q,p)-$Poin\-car\'{e}
inequality given by (\ref {qpPoin}). Similarly for an Orlicz space
$L^{\Phi },$ the $\left( L^{\Phi },L^{1}\right) -$Poincar\'{e}
inequality is the same that (\ref{poinOrli}).
\end{remark}

Obviously, we cannot expect that a $\left( X,X\right)
-$Poin\-car\'{e} inequality implies the doubling property, we shall
need to replace the space in the left-hand side in order to obtain
``a gain'', to this end we shall consider the fundamental
function\textbf{\ }of a r.i. space $X$ on $\Omega . $

Given a r.i. space $X$ on $\Omega ,$ its \textbf{fundamental
function} is defined by $\varphi _{X}(s)=\left\| \chi _{E}\right\|
_{X},$ where $E\subset
\Omega $ is an arbitrary measurable subset with $%
{\mu}%
(E)=t.$ By renorming, if necessary (see \cite{BS}), we can always
assume that $\varphi _{X}$ is concave and $\varphi _{X}(1)=1$. We
also assume in what follows that $\varphi _{X}(0)=0.$

Our main result is the following theorem.

\begin{theorem}
Let $X,Y$ be two r.i. spaces on $\Omega .$ Let $\Psi
(t)=\dfrac{\varphi _{X}(t)}{\varphi _{Y}(t)}.$ If there exists a
continuous increasing function $g:[1,\infty )\rightarrow \lbrack
1,\infty )$ with $g(1)=1,$ such that
\begin{equation*}
\Psi \left( \frac{1}{t}\right) \geq g(t),\text{ }\ (t\geq 1),
\end{equation*}
satisfying Ermakoff's condition\footnote{%
V. P. Ermakoff \cite{Er} gave in 1872 a test for convergence of
positives series based on the exponential function and the integral
test. Namely, given a continuous, positive increasing function
$g:[1,\infty )\rightarrow \lbrack 1,\infty ),$ such that
$\lim_{t\rightarrow \infty }\dfrac{tg(t)}{g(e^{t})}=a,$ then the
series $\sum_{n=1}^{\infty }\frac{1}{ng(n)}$ converges when $a<1$
and diverges when $a>1$ (a proof can be found in \cite[p. 296 and p.
298]{Kn}).
\par
We do not know any example where the limit exists but is different from $%
0,1$ or $\infty $.}
\begin{equation}
\lim_{t\rightarrow \infty }\dfrac{tg(t)}{g(e^{t})}=0.  \label{decre}
\end{equation}
Then, if the triple $\left( \Omega ,d,\mu \right) $ admits an
$\left( X,Y\right) -$ Poincar\'{e} inequality, then measure $\mu $
is doubling.
\end{theorem}

Condition (\ref{decre}) controls ``the gain'' needed on the
left-hand side
norm that allows us to deduce the doubling condition. For example, if a $%
\left( L^{\Phi },L^{1}\right) -$Poincar\'{e} inequality holds and
the Young
function satisfies $\Phi (t)\geq t\left( 1+\ln t\right) ^{\alpha }$ for $t>1$%
, then $\Psi (1/t)\geq \left( 1+\ln \frac{1}{t}\right) ^{\alpha }.$
Whence, in case that $\alpha >1$ (\ref{decre}) holds and the Theorem
applies (in particular we recover the main result of \cite{Ko}).
When $\alpha =1$ it is not possible to conclude the doubling
property (see \cite[Remark 3.1]{Ko}) (notice that in that case
condition (\ref{decre}) fails).

\section{The Proof of Theorem \label{mainth}}

Let us write $h(t)=tg(t).$ Ermakoff's condition (\ref{decre})
implies that the series $\sum_{j=1}^{\infty }\frac{1}{h(j)}$ is
convergent, let us denote by $c_{1}\ $its sum. Given
$B:=B(y,r)\subset \Omega $ set $2B:=B(y,2r)$ and define the family
of Lipschitz functions $\left\{ f_{j}\right\} _{j\in \mathbb{N}}$ in
the following way: for $j=1,$ set $r_{1}=r,$ and
\begin{equation*}
r_{j}-r_{j+1}=\dfrac{r}{2c_{1}h(j)}.
\end{equation*}

Let
\begin{equation*}
f_{j}(x):=
\begin{cases}
1, & if\;x\in B_{j+1} \\
\dfrac{r_{j}-d(x,y)}{r_{j}-r_{j+1}}, & if\;x\in B_{j}\setminus B_{j+1} \\
0, & if\;x\in \Omega \setminus B_{j}.\
\end{cases}
\end{equation*}
Also for $j\in \mathbb{N}$, define the balls $B_{j}$ as
\begin{equation*}
\frac{1}{2}B\subset B_{j}:=\left\{ x\in \Omega :d(x,y)\leq
r_{j}\right\} \subset B\subset 2B.
\end{equation*}
(the first inclusion follows from the fact that $\lim_{j\rightarrow
\infty }r_{j}=r/2).$ Then, for each $j,$
\begin{equation*}
g_{j}=\dfrac{1}{r_{j}-r_{j+1}}\chi
_{B_{j}(x)}=\dfrac{2c_{1}h(j)}{r}\chi _{B_{j}(x)}
\end{equation*}
is an upper gradient of $f_{j}.$ By the $\left( X,Y\right)
-$Poincar\'{e} inequality (\ref{poin}) applied to each $f_{j}$ on
$2B$ we get
\begin{equation}
\left\| f_{j}-\left( f_{j}\right) _{2B}\right\| _{X\left( 2B,\mu
_{2B}\right) }\leq c2r\left\| g\right\| _{Y\left( 2\sigma B,\mu
_{2\sigma B}\right) }.  \label{aaa0}
\end{equation}
By the triangle inequality,
\begin{eqnarray}
\left\| f_{j}\right\| _{X\left( 2B,\mu _{2B}\right) } &\leq &\left\|
f_{j}-\left( f_{j}\right) _{2B}\right\| _{X\left( 2B,\mu
_{2B}\right) }+\left\| \left( f_{j}\right) _{2B}\right\| _{X\left(
2B,\mu _{2B}\right) }
\label{aaa1} \\
&=&\left\| f_{j}-\left( f_{j}\right) _{2B}\right\| _{X\left( 2B,\mu
_{2B}\right) }+\left| \left( f_{j}\right) _{2B}\right|  \notag \\
&\leq &c2r\left\| g\right\| _{Y\left( 2\sigma B,\mu _{2\sigma
B}\right) }+\left| \left( f_{j}\right) _{2B}\right| \text{ \ \ (by\
(\ref{aaa0}))} \notag
\end{eqnarray}
Observe that for each fixed $j\in \mathbb{N},$ we have that
\begin{eqnarray}
\left\| g_{j}\right\| _{Y\left( 2\sigma B,\mu _{2\sigma B}\right) } &=&%
\dfrac{2c_{1}h(j)}{r}\left\| \chi _{B_{j}}\right\| _{Y\left( 2\sigma
B,\mu
_{2\sigma B}\right) }  \label{aa1} \\
&=&\dfrac{2c_{1}h(j)}{r}\varphi _{Y}\left( \frac{\mu (B_{j})}{\mu (2\sigma B)%
}\right)  \notag \\
&\leq &\dfrac{2c_{1}h(j)}{r}\varphi _{Y}\left( \frac{\mu (B_{j})}{\mu (2B)}%
\right) \text{ \ (since }\sigma \geq 1).  \notag
\end{eqnarray}
Using that $f_{j}\leq 1$ and it is supported on $B_{j}$ $\subset
2B,$ we obtain
\begin{eqnarray}
\left| \left( f_{j}\right) _{2B}\right| &\leq &\frac{1}{\mu (2B)}%
\int_{2B}\left| f_{j}\right| d\mu =\int_{0}^{1}\left( f_{j}\right)
^{\ast
}(s\mu (2B))ds  \label{aa2} \\
&\leq &\left\| \left( f_{j}\right) ^{\ast }(s\mu (2B))\right\| _{\bar{Y}%
(0,1)}\text{ (by H\"{o}lder's inequality (see \cite{BS}))}  \notag \\
&\leq &\left\| \chi _{\lbrack 0,\mu (B_{j})]}(s\mu (2B))\right\| _{\bar{Y}%
(0,1)}\text{ }  \notag \\
&=&\varphi _{Y}\left( \frac{\mu (B_{j})}{\mu (2B)}\right) .  \notag
\end{eqnarray}
Moreover, since $f_{j}=1$ in $B_{j+1}$, we get
\begin{eqnarray}
\left\| f_{j}\right\| _{X\left( 2B,\mu _{2B}\right) } &\geq &\left\|
\chi _{B_{j+1}}\right\| _{X\left( 2B,\mu _{2B}\right) }=\left\| \chi
_{\lbrack
0,\mu (B_{j+1}))}(s\mu (2B))\right\| _{\bar{X}}  \label{aa3} \\
&=&\varphi _{X}\left( \frac{\mu (B_{j+1})}{\mu (2B)}\right) .
\notag
\end{eqnarray}
Inserting the information (\ref{aa1}), (\ref{aa2}) and (\ref{aa3}) back in (%
\ref{aaa1}) we obtain
\begin{equation}
\varphi _{X}\left( \frac{\mu (B_{j+1})}{\mu (2B)}\right) \leq 4cr\dfrac{%
2c_{1}h(j)}{r}\varphi _{Y}\left( \frac{\mu (B_{j})}{\mu (2B)}\right)
. \label{bbb}
\end{equation}
At this point, for $j\in \mathbb{N},$ define
\begin{equation*}
P_{j}(B)=\dfrac{1}{Ch(j)\varphi _{Y}\left( \frac{\mu (B_{j})}{\mu (2B)}%
\right) },
\end{equation*}
where $C=8cc_{1}.$

\bigskip Using this notation we can write (\ref{bbb}) as
\begin{equation}
\varphi _{X}\left( \frac{\mu (B_{j+1})}{\mu (2B)}\right) \leq \dfrac{1}{%
P_{j}(B)}.  \label{porfin}
\end{equation}

On the other hand, (\ref{decre}) ensures that we can pick $D\geq 1$
such that
\begin{equation}
\frac{1}{eC}\dfrac{g(De^{t})}{tg(t)}>1,\;\;t\geq 1.  \label{epa1}
\end{equation}
We will show that there is a constant $\tilde{C}$ such that
\begin{equation}
P_{1}(B)=\dfrac{1}{C\varphi _{Y}\left( \frac{\mu (B)}{\mu
(2B)}\right) }\leq \tilde{C}  \label{provadoble}
\end{equation}
for all ball $B\subset \Omega ,$ which implies that $%
{\mu}%
$ is doubling.

Suppose that (\ref{provadoble}) is not satisfied, then there exists a ball $%
B\subset \Omega $ such that
\begin{equation}
P_{1}(B)=\dfrac{1}{C\varphi _{Y}\left( \frac{\mu (B)}{\mu (2B)}\right) }%
>e^{2}D.  \label{nodobla}
\end{equation}
We will show by induction that
\begin{equation}
P_{j}(B)\geq P_{1}(B)e^{j-1}.  \label{indu}
\end{equation}
The case $j=1$ is triviality true.\ Assume that $P_{j}(B)\geq
P_{1}(B)e^{j-1},$ for some $j\geq 1.$ We claim that
\begin{equation}
\varphi _{X}^{-1}(t)\leq \varphi _{Y}^{-1}\left( \frac{t}{g(\frac{1}{t})}%
\right) ,\;\;\;0<t<1.  \label{claim}
\end{equation}
Assuming momentarily the validity of (\ref{claim}), and taking into
account
that $P_{j}(B)\geq P_{1}(B)e^{j-1}>1$, it follows from (\ref{porfin}) and (%
\ref{claim}) that
\begin{equation*}
\frac{\mu (B_{j+1})}{\mu (2B)}\leq \varphi _{X}^{-1}\left( \dfrac{1}{P_{j}(B)%
}\right) \leq \varphi _{Y}^{-1}\left(
\dfrac{1}{P_{j}(B)g(P_{j}(B))}\right) .
\end{equation*}
Therefore
\begin{equation*}
\varphi _{Y}\left( \frac{\mu (B_{j+1})}{\mu (2B)}\right) \leq \dfrac{1}{%
P_{j}(B)g(P_{j}(B))}.
\end{equation*}
and so
\begin{equation*}
P_{j+1}(B)\geq \dfrac{P_{j}(B)g(P_{j}(B))}{Ch(j+1)}.
\end{equation*}
Hence,

\begin{align*}
P_{j+1}(B)& \geq \dfrac{P_{j}(B)g\left( P_{j}(B)\right) }{Ch(j+1)} \\
& \geq
\dfrac{P_{1}(B)e^{j}g(\frac{P_{1}(B)}{e}e^{j})}{Ce(j+1)g(j+1)}\text{
(by (\ref{indu})) } \\
& \geq \dfrac{P_{1}(B)e^{j}g(De^{j+1})}{Ce(j+1)g(j+1)}\text{ (by
(\ref
{nodobla}))} \\
& \geq P_{1}(B)e^{j}\text{ by (\ref{epa1}).}
\end{align*}
But now (\ref{indu}) implies that $P_{j}(B)\rightarrow \infty $ as $%
j\rightarrow \infty $, which contradicts the fact that
\begin{equation*}
P_{j}(B)=\dfrac{1}{Ch(j)\varphi _{Y}\left( \frac{\mu (B_{j})}{\mu (2B)}%
\right) }\leq \dfrac{1}{Ch(j)\varphi _{Y}\left( \frac{\mu (\frac{1}{2}B)}{%
\mu (2B)}\right) }\rightarrow 0,\;\;j\rightarrow \infty .
\end{equation*}
Thus, (\ref{provadoble}) holds true and the proof is complete.

It remains to prove (\ref{claim}). By the concavity of $\varphi
_{X}$ we get that $t\leq \varphi _{X}(t)$ if $0<t<1,$ therefore
since $g\left( 1/t\right) $ decreases we have that
\begin{equation*}
g\left( \frac{1}{\varphi _{X}(t)}\right) \leq g\left(
\frac{1}{t}\right) \leq \dfrac{\varphi _{X}(t)}{\varphi _{Y}(t)}.
\end{equation*}
Letting $t=\varphi _{X}^{-1}(s)$ we obtain
\begin{equation*}
g\left( \frac{1}{s}\right) \leq \dfrac{s}{\varphi _{Y}(\varphi
_{X}^{-1}(s))}
\end{equation*}
and so
\begin{equation*}
\varphi _{Y}(\varphi _{X}^{-1}(s))\leq \dfrac{s}{g\left(
\frac{1}{s}\right) }
\end{equation*}
i.e.
\begin{equation*}
\varphi _{Y}(t)\leq \dfrac{\varphi _{X}(t)}{g\left( \frac{1}{\varphi _{X}(t)}%
\right) }.
\end{equation*}
Whence
\begin{equation*}
t=\varphi _{Y}^{-1}(\varphi _{Y}(t))\leq \varphi _{Y}^{-1}\left( \dfrac{%
\varphi _{X}(t)}{g\left( \frac{1}{\varphi _{X}(t)}\right) }\right) ,
\end{equation*}
which is equivalent to
\begin{equation*}
\varphi _{X}^{-1}(s)\leq \varphi _{Y}^{-1}\left( \frac{s}{g\left( \frac{1}{s}%
\right) }\right) ,
\end{equation*}
as we wished to show.

\begin{remark}
The iterative arguments used in the proof are inspired in the method
used in \cite{Ko}.
\end{remark}

\begin{remark}
It is plain that if in Theorem \ref{mainth}, the r.i. space $X$
(resp. $Y)$ is replaced by any r.i. space $\hat{X}$ (resp.
$\hat{Y})$ such that has
equivalent fundamental function, i.e. there is $c\geq 1$ such that $\frac{1}{%
c}\varphi _{X}\leq $ $\varphi _{\hat{X}}\leq c\varphi _{X},$ then
the conclusion
remains true with $X$ replaced by $\hat{X}$ (resp. $Y $ replaced by $\hat{Y}%
).$

Associated with a r.i. space $X$ we get the Lorentz and
Marcinkiewicz space defined by the r.i. norms
\begin{equation*}
\left\| f\right\| _{M(X)}=\sup_{t}\left(
\frac{1}{t}\int_{0}^{t}f_{\mu }^{\ast }(s)ds\right) \varphi
_{X}(t),\text{ \ \ }\left\| f\right\| _{\Lambda (X)}=\int_{0}^{\mu
(\Omega )}f_{\mu }^{\ast }(t)d\varphi _{X}(t).
\end{equation*}
Since
\begin{equation*}
\varphi _{M(X)}(t)=\varphi _{\Lambda (X)}(t)=\varphi _{X}(t),
\end{equation*}
and (see \cite{BS})
\begin{equation*}
\Lambda (X)\subset X\subset M(X),
\end{equation*}
we have that $\left( X,Y\right) -$Poincar\'{e}'s inequalities, in
the previous Theorem can be replaced by the weaker $\left(
M(X),\Lambda (Y\right) )-$Poincar\'{e}'s inequalities.
\end{remark}

\section{Examples\label{exam}}

In this Section we shall give several examples where Theorem
\ref{mainth} can be applied.

\subsection{$L^p$-spaces}

Let $X=L^{p\sigma }(\mu )$ and $Y=L^{p}(\mu )$ where $1\leq p<\infty $ and $%
1<\sigma <\infty .$ The function $\Psi (t)=t^{\frac{1}{p}\left( \frac{1}{%
\sigma }-1\right) }$ is decreasing and considering $g(t)=\Psi (1/t)$
an
elementary computation shows that $\lim_{t\rightarrow \infty }\dfrac{t\Psi (%
\frac{1}{t})}{\Psi (e^{-t})}=0.$ Thus, if $(\Omega ,d,\mu )$ admits a $%
(L^{p\sigma },L^{p})-$Poincar\'{e} inequality, then measure $\mu $
is doubling and we recover the results of \cite{KMR} and \cite{AH}.

\begin{remark}
Let $p,q\in \lbrack 1,\infty ]$. Assume that either $1<p<\infty $
and $1\leq q\leq \infty ,$ or $p=q=1$, or $p=q=\infty $. Then the
functional defined as
\begin{equation*}
\left\| f\right\| _{L^{p,q}(\Omega ,\mu )}=\left\| t^{\frac{1}{p}-\frac{1}{q}%
}f_{\mu }^{\ast }(t)\right\| _{L^{q}(\left( 0,\mu (\Omega )\right)
,ds)}
\end{equation*}
is equivalent to a r.i. function norm. The corresponding r.i. space
is called a \textbf{Lorentz space}.

Assume now that either $1<p<\infty $ and $1\leq q\leq \infty $ and
$\alpha
\in \mathbb{R}$, or$p=1,q=1$ and $\alpha \geq 0$ or $p=q=\infty $ and $%
\alpha \leq 0,$ or $p=\infty ,$ $1\leq q\leq \infty $ and $\alpha
+1/q<0.$ Then also the functional given by
\begin{equation*}
\left\| f\right\| _{L^{p,q,\alpha }(\Omega ,\mu )}=\left\| t^{\frac{1}{p}-%
\frac{1}{q}}f_{\mu }^{\ast }(t)(1+\ln ^{+}\frac{1}{t})^{\alpha
}\right\| _{L^{q}(\left( 0,\mu (\Omega )\right) ,ds)}
\end{equation*}
is equivalent to a r.i. function norm. The r.i. space built upon
this function norm is called a \textbf{Lorentz-Zygmund} space.

Lorentz spaces and Lorentz-Zygmund spaces has the same fundamental
function that Lebesgue spaces, hence the same result holds for these
spaces.
\end{remark}

\subsection{Orlicz spaces}

When working in Orlicz spaces, it is useful to establish conditions
in terms of Young's functions instead of fundamental functions, in
the next result we state the version of Theorem \ref{mainth} in this
sense.

\begin{theorem}
Let $L^{A}\left( \Omega ,\mu \right) $ and $L^{\hat{A}}\left( \Omega
,\mu
\right) $ be two Orlicz spaces on $\Omega $ with Young functions $A$ and $%
\hat{A}$. Let $g:[1,\infty )\rightarrow \lbrack 1,\infty )$ be a
continuous increasing function with $g(1)=1,$ such that
\begin{equation}
A\left( tg(t)\right) \leq \hat{A}\left( t\right) ,\text{ \ \ }t>1.
\label{OrliczYoung}
\end{equation}
that satisfies (\ref{decre}). Then, if the triple $\left( \Omega
,d,\mu \right) $ admits a $\left( L^{\hat{A}},L^{A}\right)-$
Poincar\'{e} inequality, then measure $\mu $ is doubling.
\end{theorem}

\begin{proof}
From $\hat{A}\left( t\right) \geq A\left( tg(t)\right) $ we obtain
\begin{equation*}
tg(t)\leq A^{-1}\left( \hat{A}\left( t\right) \right) .
\end{equation*}
Since $\varphi _{L^{A}}(t)=\frac{1}{A^{-1}(1/t)}$ (see \cite{BS}),
we can write the above inequality as
\begin{equation*}
tg(t)\leq \frac{1}{\varphi _{L^{A}}\left( \frac{1}{\hat{A}\left( t\right) }%
\right) }.
\end{equation*}
On the other hand, since $\varphi _{L^{\hat{A}}}^{-1}(\frac{1}{t})=\frac{1}{%
\hat{A}\left( t\right) },$ we get that
\begin{equation*}
tg(t)\leq \frac{1}{\varphi _{L^{A}}\left( \varphi _{L^{\hat{A}}}^{-1}(\frac{1%
}{t})\right) },
\end{equation*}
which implies
\begin{equation*}
\varphi _{L^{\hat{A}}}^{-1}(t)\leq \varphi _{L^{A}}^{-1}\left( \frac{t}{%
g\left( \frac{1}{t}\right) }\right) ,
\end{equation*}
Now with the same argument used in the proof of Theorem \ref{mainth} from (%
\ref{claim}) we finish the proof.
\end{proof}

In the particular case that $(\Omega ,d,\mu )$ admits a $\left(
L^{A},L^{1}\right) -$Poincar\'{e} inequality with $A(t)\geq t(1+\ln
t)^{\alpha }$ for $t>1,$ we recover \cite[Theorem 2.4]{Ko}.

\subsection{General Case}

The upper and lower Zippin indices associated with a r.i. space $X$
are defined by
\begin{equation*}
\overline{\beta }_{X}=\inf\limits_{s>1}\dfrac{\ln M_{X}(s)}{\ln
s}\text{ \ \ and \ \ }\underline{\beta
}_{X}=\sup\limits_{s<1}\dfrac{\ln M_{X}(s)}{\ln s},
\end{equation*}
where
\begin{equation*}
M_{X}(s)=\sup\limits_{t>0}\dfrac{\varphi _{X}(ts)}{\varphi _{X}(t)},\text{ }%
s>0.
\end{equation*}
It is known that (see \cite[p. 272]{Zip})
\begin{equation*}
0\leq \underline{\beta }_{X}\leq \overline{\beta }_{X}\leq 1,
\end{equation*}
and that for any $\varepsilon >0$ there is $\delta =\delta
(\varepsilon )$ such that the following inequalities are satisfied
\begin{equation}
s^{\underline{\beta }_{X}}<M_{X}(s)<s^{\underline{\beta }_{X}-\varepsilon },%
\text{ \ \ }0<s<\delta,  \label{zipinf}
\end{equation}
\begin{equation}
s^{\overline{\beta }_{X}}<M_{X}(s)<s^{\overline{\beta }_{X}-\varepsilon },%
\text{ \ \ }s>1/\delta.  \label{zipsup}
\end{equation}
Let $X,Y$ two r.i. spaces, then
\begin{equation*}
\frac{\varphi _{Y}(t)}{\varphi _{X}(t)}\leq \frac{\varphi
_{Y}(st)}{\varphi
_{X}(st)}\sup\limits_{t>0}\frac{\varphi _{X}(st)}{\varphi _{X}(t)}%
\sup\limits_{t>0}\frac{\varphi _{Y}(t)}{\varphi
_{Y}(st)}=\frac{\varphi _{Y}(st)}{\varphi
_{X}(st)}M_{X}(s)M_{Y}(1/s).
\end{equation*}
Pick $\varepsilon >0,$ letting $s=1/t,$ and combining (\ref{zipinf}) and (%
\ref{zipsup}) we get (for $t$ small enough)
\begin{equation*}
\frac{\varphi _{Y}(t)}{\varphi _{X}(t)}\leq
M_{X}(\frac{1}{t})M_{Y}(t)\leq t^{{\overline{\beta
}_{X}}-\underline{\beta }_{Y}},
\end{equation*}
thus
\begin{equation*}
\Psi (t)=\frac{\varphi _{X}(t)}{\varphi _{Y}(t)}\geq t^{\underline{\beta }%
_{Y}-\bar{\beta}_{X}}.
\end{equation*}
Therefore, if $\underline{\beta }_{Y}<\bar{\beta}_{X}$ and $\left(
\Omega ,d,\mu \right) $ admits a $\left( X,Y\right) -$ Poincar\'{e}
inequality, then $\mu $ is doubling.

In certain sense, Zippin indices (see \cite{Zip}) indicate us the
`position'' of the space, with respecte to $L^{p}$ spaces, notice that $%
\underline{\beta }_{L^{p}}=\overline{\beta }_{L^{p}}=\frac{1}{p}$.

We can also apply our Theorem when Zippin indices of the involved
spaces coincide, but in this case need to ensure that $\varphi _{Y}$
decreases slightly more rapidly that $\varphi _{Y}(t),$ for example,
assuming that there exists $b:(0,1)\rightarrow (0,\infty )$ a slowly
varying decreasing function (i.e.  for each $\varepsilon >0$, the
function $t^{\varepsilon }b(t)$ is equivalent to a increasing
function and $t^{-\varepsilon }b(t)$ is equivalent to a decreasing
function (see \cite{BGT})), so that
\begin{equation*}
\frac{\varphi _{X}(t)}{\varphi _{Y}(t)}\geq b(t).
\end{equation*}
Then, if $b(1/t)$ satisfies Ermakoff's condition (\ref{decre}) and
$\left( \Omega ,d,\mu \right) $ admits a $\left( X,Y\right) -$
Poincar\'{e} inequality, then $\mu $ is doubling.

We finish the paper giving some examples of decreasing slowly
varying functions that satisfies Ermakoff's condition

\begin{example}
Let $n\in \mathbb{N}$. The iterated logarithmic on $(0,1)$ are
defined by
\begin{equation*}
\begin{array}{l}
L_{1}(t)=\ell (t)=1+\ln \frac{1}{t} \\
L_{n+1}(t)=\ell (L_{n}(t)),\text{ \ }n\geq 2.
\end{array}
\end{equation*}

The following functions are decreasing slowly varying functions that
satisfy Ermakoff's condition:

\begin{enumerate}
\item  Let $k<m,$ $\left( k,m\in \mathbb{N}\right) $%
\begin{equation*}
c_{k,m}(t)=\exp \left( \frac{L_{k}(t)}{L_{m}(t)}\right)
\end{equation*}

\item  Let $k\geq 1,$ and $\hat{\alpha}_{k}=\left( \alpha _{1},\alpha
_{2},\cdots ,\alpha _{k}\right) $ where $0<\alpha _{j}<1,$ $\left(
j\leq
k\right) $%
\begin{equation*}
d_{k}(t)=\exp \left( \left( L_{1}(t)\right) ^{\alpha _{1}}\left(
L_{2}(t)\right) ^{\alpha _{2}}\cdots \left( L_{k}(t)\right) ^{\alpha
_{k}}\right) .
\end{equation*}

\item  Let $m\in \mathbb{N}$ and $\alpha >1$%
\begin{equation*}
b_{m,\alpha }(t)=\left( \prod_{j=1}^{m-1}L_{j}(t)\right) \left(
L_{n}(t)\right) ^{\alpha }.
\end{equation*}
(Notice that if $\alpha \leq 1$ Ermakoff's condition is not
fulfilled).
\end{enumerate}
\end{example}

\end{document}